\documentclass[12]{article}
\usepackage[ansinew]{inputenc}
\usepackage{amsmath,amsthm,amsfonts,enumerate}

\newtheorem{theorem}{Theorem}[section]
\newtheorem{lemma}[theorem]{Lemma}
\newtheorem{proposition}[theorem]{Proposition}
\newtheorem{corollary}[theorem]{Corollary}

\newtheorem{result}[theorem]{Result}
\newtheorem{remark}[theorem]{Remark}

\newcommand{\A}{{\cal A}}
\newcommand{\B}{{\cal B}}

\DeclareMathOperator{\tw}{tw}

\date{}
\title{On the treewidth of generalized Kneser graphs}
\author{Klaus Metsch\thanks{Justus-Liebig-Universit\"{a}t, Mathematisches Institut,
Arndtstra{\ss}e 2, D-35392 Gie{\ss}en}}

\begin{document}
\maketitle

\begin{abstract}
The generalized Kneser graph $K(n,k,t)$ for integers $k>t>0$ and $n>2k-t$ is the graph whose vertices are the $k$-subsets of $\{1,\dots,n\}$ with two vertices adjacent if and only if  they share less than $t$ elements. We determine the treewidth of the generalized Kneser graphs $K(n,k,t)$ when $t\ge 2$ and $n$ is sufficiently large compared to $k$. The imposed bound on $n$ is a significant improvement of a previously known bound. One consequence of our result is the following. For each integer $c\ge 1$ there exists a constant $K(c)\ge 2c$ such that $k\ge K(c)$ implies for $t=k-c$ that
$$\tw(K(n,k,t))=\binom{n}{k}-\binom{n-t}{k-t}-1$$ if and only if $n\ge (t+1)(k+1-t)$ .

\end{abstract}



\section{Introduction}

In this paper a \emph{graph} $\Gamma$ is a pair $(X,E)$ where $X$ is a finite non-empty set and $E$ is a set of subsets of cardinality two of $X$. The elements of $X$ are called \emph{vertices} and the elements of $E$ are called \emph{edges}. We write $X=V(\Gamma)$. $\Gamma$ is called empty if it has no edges.

A \emph{tree decomposition} of a graph $\Gamma$ is a pair $(T,B)$ where $T$ is a tree and $B=(B_t:t\in V(T))$ is a collection of subsets $B_t$ of $V(\Gamma)$, indexed by the vertices of $T$, such that
\begin{enumerate}
\item every edge $\{u,v\}$ of $\Gamma$ is contained in $B_t$ for some $t\in V(T)$, and
\item for each $v\in V(\Gamma)$, the graph induced by $T$ on $\{t\in V(T)\mid v\in B_t\}$ is connected and non-empty.
\end{enumerate}
The \emph{width} of such a tree decomposition is the number $\max\{|B_t|-1\mid t\in V(T)\}$, and the \emph{treewidth} $\tw(\Gamma)$ of a graph $\Gamma$ is the smallest width of its tree decompositions. The treewidth of a graph measures how treelike a graph is. For example, the treewidth of a non-empty tree is one and the treewidth of a graph on $n$ vertices is at most $n-1$ with equality if and only if the graph is complete. There is a vast literature on the treewidth of graphs, see \cite{c3,c5,c6,c2,Liu&Ciao&Lu,c4,c1} for some recent ones,  and there are applications. A famous one is the use of treewidth by Robertson and Seymour \cite{Robertson&Seymour} in their minor theorem.

For integers $n,k,t$ with $k>t\ge 1$ and $n>2k-t$, the \emph{generalized Kneser graph} $K(n,k,t)$ is the graph whose vertices are the $k$-element subsets of the set $[n]:=\{1,\dots,n\}$ with two vertices $K$ and $K'$ adjacent if and only if $|K\cap K'|<t$. The condition $n>2k-t$ ensures that the graph is non-empty. If $t=1$, these graphs are called \emph{Kneser graph}s and are denoted by $K(n,k)$.

It was proved by Harvey and Wood \cite{Harvey&Wood}  that the treewidth of $K(n,k)$ is equal to $\binom{n}{k}-\binom{n-1}{k-1}-1$ for $n\ge 4k^2-4k+1$ and $k\ge 3$. More recently,  Liu,  Ciao and Lu proved the following.

\begin{theorem}[\cite{Liu&Ciao&Lu}]
For integers $n,k,t$ with $k>t\ge 2$ and
\begin{align}\label{Liubound}
n\ge 2(k-t)(t+1)\binom{k}{t}+k+t+1
\end{align}
we have $\tw(K(n,k,t))=\binom{n}{k}-\binom{n-t}{k-t}-1$.
\end{theorem}

The hard part in this theorem is to prove the lower bound for the treewidth. For the upper bound, the authors of \cite{Liu&Ciao&Lu} also used \eqref{Liubound} but in fact this holds in general, which is our first result.

\begin{theorem}\label{upperbound}
Let $n,k,t$ be positive integers with $n>2k-t$ and $k>t>0$.
\begin{enumerate}[\rm (a)]
\item  $\tw(K(n,k,t))\le \binom{n}{k}-\binom{n-t}{k-t}-1$.
\item If $n<(t+1)(k+1-t)$, then the bound in (a) is not tight.
\end{enumerate}
\end{theorem}

The reason for (a) not to be tight when $n<(t+1)(k+1-t)$ is that in this situation the generalized Kneser graph has independent sets that are larger than the so called point pencils, see Section 2. Up to my knowledge, it is open problem whether equality holds in (a) for all $n\ge (t+1)(k+1-t)$. We can give an affirmative answer for some parameter sets as follows.

\begin{theorem}\label{main1}
For each integer $c\ge 1$, there exists an integer $K(c)$ such that $k\ge K(c)$ and $t=k-c$ implies that
\begin{align}\label{twbound}
\tw(K(n,k,k-c))=\binom{n}{k}-\binom{n-t}{k-t}-1
\end{align}
for all $n$ with $n\ge (t+1)(k+1-t)$.
\end{theorem}

We also give $K(c)$ more explicitly in Corollary \ref{firstcorollary}. For general $k$ and $t$ we can improve the above result of Liu,  Ciao and Lu as follows.

\begin{theorem}\label{main2}
For integers $n,k,t$ with $k>t>1$ and
\begin{align}
n\ge \begin{cases}
        6k(k+1-t)(k-t) & \text{if $2\le t\le 16$,}
        \\
        \frac{t-1}{\ln(t)}k(k+1-t)(k-t) & \text{if $t\ge 17$}
    \end{cases}
\end{align}
we have
\begin{align*}
\tw(K(n,k,t))=\binom{n}{k}-\binom{n-t}{k-t}-1.
\end{align*}
\end{theorem}

The proofs in the present paper follows the lines of the proofs in \cite{Liu&Ciao&Lu} by improving and simplifying their arguments. In particular, our proof avoids results of the $t$-shadow of families of sets.

\section{Proof}

For every graph $\Gamma$ its maximum vertex degree is denoted by $\Delta(\Gamma)$ and is called its \emph{maximum degree}. The cardinality of a largest independent set of $\Gamma$ is denoted by $\alpha(\Gamma)$ and is called the \emph{independence number} of the graph. There is a connection between treewidth, maximum degree and independence number.

To see this, consider a graph $\Gamma$ and an independent set $A$ of $\Gamma$ of size $\alpha(\Gamma)$. Put $S=V(\Gamma)\setminus A$. Then $T:=(A\cup \{S\},\{\{a,S\}\mid a\in A\})$ is a tree (in fact a star) and defining $B_S=S$ and $B_a=\{a\}\cup \Gamma_1(a)$ for $a\in A$ results in the tree decomposition $(T,(B_i)_{i\in A\cup\{S\}})$ of $\Gamma$. Since $|B_S|=|S|=|V(\Gamma)|-\alpha(\Gamma)$ and $|B_a|\le \Delta(\Gamma)+1$, it follows that
\begin{align}\label{twboundwithDelta}
\tw(\Gamma)\le\max\{\Delta(\Gamma),|V(\Gamma)|-\alpha(\Gamma)-1\}.
\end{align}
This was proved in \cite{Harvey&Wood}.

We will apply this to generalized Kneser graphs. The independence number of the generalized Kneser graph $K(n,k,t)$ is at least $\binom{n-t}{k-t}$ since the $k$-element subsets of an $n$-set containing a given $t$-set is an independent set of this size. We now compare $\Delta(\Gamma)$ and $|V(\Gamma)|-\alpha(\Gamma)-1$ for generalized Kneser graphs $\Gamma$.

\begin{lemma}
For positive integers $n>k\ge t$ we have
\begin{align*}
\binom{n-t}{k-t}+(k-t)t\binom{n-k}{k-t}+\sum_{i=0}^{t-1}\binom{k}{i}\binom{n-k}{k-i}\le\binom{n}{k}.
\end{align*}
\end{lemma}
\begin{proof}
The sum is equal to the cardinality of the set $T_1$ that consists of all $k$-subsets of $[n]$ that have at most $t-1$ elements in $[k]$. Also $\binom{n-t}{k-t}$ is equal to the cardinality of the set $T_2$ that consists of all $k$-subsets of $[n]$ that contain $[t]$. Finally, $(k-t)t\binom{n-k}{k-t}$ is the number of $k$-subsets of $[n]$ that contain exactly $t-1$ elements of $[t]$ and one further element of $[k]$. As the three sets $T_i$ are mutually disjoint, we have that $|T_1|+|T_2|+|T_3|$ is at most the number $\binom{n}{k}$ of $k$-subsets of $[n]$.
%
%
\end{proof}

\begin{corollary}\label{corDelta}
For a generalized Kneser graph $\Gamma=K(n,k,t)$ with $n>2k-t$ and $k>t>0$ we have
\begin{align*}
\Delta(\Gamma)\le |V(\Gamma)|-\binom{n-t}{k-t}-(k-t)t\binom{n-k}{k-t}.
\end{align*}
\end{corollary}
\begin{proof}
We have $\Delta(\Gamma)=\sum_{i=0}^{t-1}\binom{k}{i}\binom{n-k}{k-i}$. In fact for every $k$-subset $K$ of $[n]$, this is the number of $k$-subsets of $[n]$ that meet $K$ in at most $t-1$ elements. The statement follows therefore from the lemma.
\end{proof}

\begin{proposition}\label{upperbound}
Let $n,k,t$ be positive integers with $n>2k-t$ and $k>t>0$.
\begin{enumerate}[\rm (a)]
\item  $\tw(K(n,k,t))\le \binom{n}{k}-\binom{n-t}{k-t}-1$.
\item $n<(t+1)(k+1-t)$, then the bound in (a) is not tight.
\end{enumerate}
\end{proposition}
\begin{proof}
Put $\Gamma=K(n,k,t)$. We will apply \eqref{twboundwithDelta}. Since $n>2k-t$ and $k>t>0$ we have $\Delta(\Gamma)\le \binom{n}{k}-\binom{n-t}{k-t}-2$ from Corollary \ref{corDelta}.

We have $|V(K(n,k,t))|=\binom{n}{k}$ and $\alpha(\Gamma)\ge \binom{n-t}{k-t}$, as noticed above. Hence
$|V(\Gamma)|-\alpha(\Gamma)\le \binom{n}{k}-\binom{n-t}{k-t}$. Thus (a) follows from \eqref{twboundwithDelta}.

Now suppose that $n<(t+1)(k+1-t)$. Let $A$ be the set consisting of all $k$-subsets of $[n]$ that have at least $t+1$ elements in $[t+2]$. This is an independent set of $\Gamma$ and thus $\alpha(\Gamma)\ge|A|$. Also, since $n<(t+1)(k-1-t)$, it is easy to see that $|A|>\binom{n-t}{k-t}$. Hence $|V(\Gamma)|-\alpha(\Gamma)<\binom{n}{k}-\binom{n-t}{k-t}$, and thus also (b) follows from \eqref{twboundwithDelta}.
\end{proof}

It was proved in \cite{Liu&Ciao&Lu} that the upper bound is sharp when $n$ is sufficiently large compared to $k$ and $t$. We will improve this result by weakening the required bound on $n$ significantly. As in \cite{Liu&Ciao&Lu} and \cite{Harvey&Wood} we use a result of Robertson and Seymour on separators. For a real number $p$ with $\frac23\le p<1$ a \emph{$p$-separator} of a graph $\Gamma$ is subset $X$ of the vertex set $V(\Gamma)$ of $\Gamma$ such that every component of $\Gamma\setminus X$ has at most $p|V(\Gamma)\setminus X)|$ vertices.

\begin{result}[\cite{Robertson&Seymour}]\label{Robertson&Seymour}
Let $\Gamma$ be a finite graph and $p$ a real number with $\frac23\le p<1$. Then $\Gamma$ has a $p$-separator $X$ with $|X|\le \tw(\Gamma) + 1$.
\end{result}

A second ingredient of our proof is the result of Wilson on the independence number of generalized Kneser graphs.

\begin{result}[\cite{Wilson}]\label{Wilson}
For integers $k>t\ge 1$ and $n\ge (t+1)(k+1-t)$, we have $\alpha(K(n,k,t))=\binom{n-t}{k-t}$.
\end{result}

As we have already noticed previously the independence number is larger when $n<(t+1)(k+1-t)$.

\begin{lemma}\label{lemma}
Suppose $n,k,t$ are positive integers with $k>t$ and $n\ge t+\frac12(k+1-t)(k-t)$. Define the function
\[
f:\{r\in\mathbb{Z}\mid 0\le r\le t-1\}\to\mathbb{R},\ f(r):=\binom{k-r}{t-r} \binom{n-2t+r}{k-2t+r}.
\]
Then $f$ is monotone increasing.
\end{lemma}
\begin{proof}
For integers $r$ with $0\le r\le t-2$, it is easy to see that $f(r)\le f(r+1)$ is equivalent to
\[
(k-r)(k-2t+r+1)\le (n-2t+r+1)(t-r).
\]
This in turn can be written as $(n-t)(t-r)\ge (k-t)(k+1-t)$.
In view of the assumed lower bound on $n$ and $r\le t-2$, this is true.
\end{proof}

\begin{lemma}\label{basicresultonseparator}
Suppose $n,k,t$ are integers with $k>t>0$ and $p$ is a real number with $\frac23\le p<1$ such that
    \begin{align*}
    n&\ge (t+1)(k+1-t),\\
    n&\ge t+\frac12(k+1-t)(k-t),\ \text{and}\\
    (1-p)\binom{n-t}{k-t}&\ge \sum_{s=1}^{t}\binom{t-1}{s-1}\binom{k+1-t}{s}\binom{k-t+s}{s} \binom{n-t-s}{k-t-s}.
    \end{align*}
Then every $p$-separator of $K(n,k,t)$ has at least $\binom{n}{k}-\binom{n-t}{k-t}$ elements.
\end{lemma}
\begin{proof}
Assume on the contrary that there exits a $p$-separator $X$ with $|X|<\binom{n}{k}-\binom{n-t}{k-t}$. Then $U:=V(K(n,k,t))\setminus X$ satisfies
$|U|>\binom{n-t}{k-t}$. The components of $K(n,k,t)\setminus X$ have each at most $p|U|$ elements, and hence there is a union $\A$ of components with $(1-p)|U|\le |\A|\le p|U|$ (this is clear if some component has at least $(1-p)|U|$ vertices and otherwise it follows from $p\ge 2(1-p)$). Let $\B=U\setminus\A$ be the union of the remaining components, so that also $(1-p)|U|\le |\B|\le p|U|$. As $\A$ and $\B$ are unions of components of the graph $K(n,k,t)\setminus X$, we have $|A\cap B|\ge t$ for all $A\in\A$ and $B\in\B$.

Since $n\ge (t+1)(k+1-t)$, Results \ref{Wilson} shows that $|U|>\alpha(K(n,k,t))$. Hence $U$ contains adjacent vertices, that is there exists two elements $A_1,A_2\in U$ that intersect in less than $t$ elements. We may assume that $A_1,A_2\in A$. Put $s=|A_1\cap A_2|<t$. For every $t$-subset $Y$ of $A_1$ we put $\B_Y:=\{B\in\B\mid Y\subseteq B\}$. The following observation is from \cite{Liu&Ciao&Lu}.

Claim: If $Y$ is a $t$-subset of $A_1$ and if $r:=|Y\cap A_2|$, then $|\B_Y|\le f(r)$.
\\
This can be seen as follows. If $B\in\B_Y$, then $|A_2\cap B|\ge t$ and hence $|(A_2\setminus Y)\cap B|\ge t-r$. On the other hand, for each of the $\binom{k-r}{t-r}$ subsets $Z$ of size $t-r$ of $A_2\setminus Y$, there are exactly $\binom{n-2t+r}{k-2t+r}$ $k$-subsets of $[n]$ that contain $Z\cup Y$. It follows that $|\B_Y|\le \binom{k-r}{t-r}\binom{n-2t+r}{k-2t+r}=f(r)$ establishing the claim.

Define $S:=A_1\cap A_2$, so that $s=|S|\le t-1$. Since $|A_1\cap B|\ge t$ for all $B\in\B$ we have
\begin{align}\label{Bistkleinergleich}
\bigcup_{Y\in{A _1\choose t}}\B_Y=\B.
\end{align}
Fix any subset $T$ of $A_1$ with $S\subseteq T$ and $|T|=t-1$. Then
\begin{align*}
|\B|&\le \sum_{Y\in{A _1\choose t}}|\B_Y|\le \sum_{Y\in{A _1\choose t}}f(|Y\cap S|)
\\
&\le \sum_{Y\in{A _1\choose t}}f(|Y\cap T|)
\le \sum_{r=0}^{t-1}\binom{t-1}{r}\binom{k+1-t}{t-r}f(r)
\end{align*}
where the first inequality follows from \eqref{Bistkleinergleich}, the second from the above claim, the third from Lemma \ref{lemma},
and the fourth by counting for $0\le r\le t-1$ how many $t$-subsets of $A_1$ share exactly $r$ elements with $T$.
Using $|\B|\ge (1-p)|U|$ and $|U|>\binom{n-t}{k-t}$ we find
\begin{align}\label{eqnsderf}
(1-p)\binom{n-t}{k-t}<\sum_{r=0}^{t-1}\binom{t-1}{r}\binom{k+1-t}{t-r}\binom{k-r}{t-r} \binom{n-2t+r}{k-2t+r}.
\end{align}
If we substitute $r=t-s$ the resulting inequality contradicts the hypothesis of the present lemma.
\end{proof}

\begin{theorem}\label{cor1}
Suppose $n,k,t$ are integers with $k>t>0$ such that
\begin{align}
n&\ge (t+1)(k+1-t),\label{eqns1}
\\
n&\ge t+\frac12(k+1-t)(k-t),\ \text{and}\label{eqns2}
\\
\frac13\binom{n-t}{k-t}&\ge \sum_{s=1}^{\min\{t,k-t\}}\binom{t-1}{s-1}\binom{k+1-t}{s}\binom{k-t+s}{s} \binom{n-t-s}{k-t-s}.\label{eqns3}
\end{align}
Then $\tw(K(n,k,t))=\binom{n}{k}-\binom{n-t}{k-t}-1$.
\end{theorem}
\begin{proof}
By Lemma \ref{basicresultonseparator} every $\frac23$-separator of $K(n,k,t)$ has at least $\binom{n}{k}-\binom{n-t}{k-t}$ elements. Result \ref{Robertson&Seymour} implies thus that $\tw(K(n,k,t))\ge \binom{n}{k}-\binom{n-t}{k-t}-1$. Proposition \ref{upperbound} gives equality.
\end{proof}

It remains to analyze inequality \eqref{eqns3}. If $k-t$ is sufficiently small, we find the following.

\begin{corollary}\label{firstcorollary}
For integers $c\ge 1$ define
\begin{align*}
K(c)=c-1+3\sum_{s=1}^{c}\binom{c-1}{s-1}\binom{c+1}{s}\binom{c+s}{s}\frac{1}{c^{s-1}}.
\end{align*}
Then for all integers $k$ and $t$ with $k\ge K(c)$ and $t=k-c$ we have
\begin{align}\label{twbound}
\tw(K(n,k,k-c))=\binom{n}{k}-\binom{n-t}{k-t}-1
\end{align}
for all $n$ with $n\ge (t+1)(k+1-t)$.
\end{corollary}
\begin{proof}
Consider integers $c,k,t,n$ with $c\ge 1$, $k\ge K(c)$, $t=k-c$ and
\begin{align}\label{eqnboundn}
n\ge (t+1)(k+1-t)=(t+1)(c+1).
\end{align}
Since $K(c)>2c$, then $k>2c$ and $t=k-c\ge c+1\ge 2$. Also \eqref{eqnboundn} implies $n>2k-t$, so that $n,k,t$ are parameters of a generalized Kneser graph $K(n,k,t)$. As $k>2c$, then \eqref{eqnboundn} implies \eqref{eqns2}. Assume that \eqref{twbound} is not true. Then Theorem \ref{cor1} shows that
\begin{align}
\frac13\binom{n-t}{c}&<\sum_{s=1}^{c}\binom{t-1}{s-1}\binom{c+1}{s}\binom{c+s}{s} \binom{n-t-s}{c-s}
\\
\Rightarrow
\frac13\frac{(n-t)!}{c!}&\le \sum_{s=1}^{c}\frac{(t-1)!}{(t-s)!(s-1)!}\binom{c+1}{s}\binom{c+s}{s}\frac{(n-t-s)!}{(c-s)!}
\\
\Rightarrow
n-t&\le \sum_{s=1}^{c}\frac{3c!}{(c-s)!(s-1)!}\binom{c+1}{s}\binom{c+s}{s}\frac{(t-1)!}{(t-s)!}\frac{(n-t-s)!}{(n-t-1)!}
\label{eqnfinal1}
\end{align}
From \eqref{eqnboundn} we find $n-t\ge t$ and $n-t-1\ge (t+1)c$. This proves the inequalities in
\begin{align*}
\frac{(t-1)!}{(t-s)!}\frac{(n-t-s)!}{(n-t-1)!}=\prod_{i=1}^{s-1}\frac{t-i}{n-t-i}\le \frac{(t-1)^{s-1}}{(n-t-1)^{s-1}}
\le \frac{1}{c^{s-1}}.
\end{align*}
On the left hand side of \eqref{eqnfinal1} we use that $n-t>(t+1)c=(k+1-c)c$ and find
\begin{align}
k+1-c&<\sum_{s=1}^{c}\frac{3(c-1)!}{(c-s)!(s-1)!}\binom{c+1}{s}\binom{c+s}{s}\frac{1}{c^{s-1}}
\end{align}
and hence $k<K(c)$. Since we assumed $k\ge K(c)$, this is a contradiction coming from the assumption that \eqref{twbound} is not true. Therefore \eqref{twbound} is true.
\end{proof}

Corollary \ref{firstcorollary} and Proposition \ref{upperbound} prove Theorem \ref{main1}.

\begin{remark}\rm
The value $K(c)$ was chosen in such a way that \eqref{eqnfinal1} is not satisfied. For fixed small $c$ one can determine $K'(c)$ explicitly such that \eqref{eqnfinal1} is not satisfied iff $k\ge K'(c)$ and hence \eqref{twbound} is satisfied for all  $n,k,t$ with $k\ge K'(c)$, $t=k-c$ and $n\ge (t+1)(k+1-t)$. For example one finds $K'(1)=12$, $K'(2)=54$, $K'(3)=195$ and $K'(4)=626$. For $c=1$, a better result was proved in \cite{Liu&Ciao&Lu}.
\end{remark}

\begin{corollary}\label{lastcor}
Suppose $n,k,t$ are integers with $k>t>1$ and
\begin{align}\label{hyponerder}
n\ge \begin{cases}
     t+\frac{1}{\ln(t)}(t-1)k(k+1-t)(k-t) & \text{if $t\ge 17$}
     \\
          t+6k(k+1-t)(k-t) & \text{if $2\le t\le 16$}
     \end{cases}
\end{align}
Then $\tw(K(n,k,t))=\binom{n}{k}-\binom{n-t}{k-t}-1$.
\end{corollary}
\begin{proof}
Define \begin{align}\label{def_c}
c:=\frac{(t-1)k(k-t)(k+1-t)}{n-t}.
\end{align}
Assume the statement is wrong. Then Theorem \ref{cor1} shows that
\begin{align}\label{eqcor2a}
\frac13\binom{n-t}{k-t}\le \sum_{s=1}^{t}\binom{t-1}{s-1}\binom{k+1-t}{s}\binom{k-t+s}{s} \binom{n-t-s}{k-t-s}.
\end{align}
For every integer $s$ with $1\le s\le t$ we have
\begin{align*}
\binom{n-t-s}{k-t-s}&=\binom{n-t}{k-t}\prod_{i=0}^{s-1}\frac{k-t-i}{n-t-i}\le \left(\frac{k-t}{n-t}\right)^s\binom{n-t}{k-t}.
\end{align*}
Using this and $\binom{a}{b}\le a^b/b!$ for binomial coefficients with $a,b\ge 0$ in \eqref{eqcor2a}, we find \begin{align*}
\frac13&\le \sum_{s=1}^{t}\binom{t-1}{s-1}\binom{k+1-t}{s}\binom{k-t+s}{s}\left(\frac{k-t}{n-t}\right)^s
\\
&\le \sum_{s=1}^{t}\frac{(t-1)^s(k+1-t)^s(k-t+s)^s}{(t-1)(s-1)!s!s!}\cdot   \left(\frac{k-t}{n-t}\right)^s.
\end{align*}
Using the definition of $c$ in \eqref{def_c}, this implies that
\begin{align}\label{eqcor2c}
\frac{t-1}{3}
&\le \sum_{s=1}^{t}\frac{c^s}{(s-1)!s!s!}.
\end{align}

Case 1. We have $2\le t\le 16$. Then one can check for each possible value of $t$ that \eqref{eqcor2c} implies that $c>\frac16(t-1)$. Using \eqref{def_c}, this contradicts \eqref{hyponerder}.

Case 2. We have $t\ge 17$. By hypotheses and \eqref{def_c} we then have $c\le\ln(t)$ and therefore
\begin{align*}
\frac{t-1}{3}&\le \sum_{s=1}^{t}\frac{c^s}{(s-1)!s!s!}
\\
&\le c+\frac{c^2}{4}+\frac{1}{2!3!}\sum_{s=3}^\infty\frac{c^s}{s!}
\\
&\le c+\frac{c^2}{4}+\frac{1}{2!3!}(e^c-1)
\\
&\le \ln t +\frac14(\ln t)^2+\frac1{12}(t-1).
\end{align*}
This implies that
\begin{align*}
t-1&\le4\ln t +(\ln t)^2.
\end{align*}
For $t=24$ and hence for all $t\ge 24$, this is a contradiction.
For $17\le t\le 23$ one can check easily that \eqref{eqcor2c} implies that $c>\ln(t)$, which is a contradiction.
\end{proof}

Corollary \ref{lastcor} proves Theorem \ref{main2}.


\begin{thebibliography}{10}

\bibitem{c3}
David Eppstein, Daniel Frishberg, and William Maxwell.
\newblock On the treewidth of {H}anoi graphs.
\newblock {\em Theoret. Comput. Sci.}, 906:1--17, 2022.

\bibitem{Harvey&Wood}
Daniel~J. Harvey and David~R. Wood.
\newblock Treewidth of the {K}neser graph and the {E}rd{\H{o}}s-{K}o-{R}ado
  theorem.
\newblock {\em Electron. J. Combin.}, 21(1):Paper 1.48, 11, 2014.

\bibitem{c5}
Daniel~J. Harvey and David~R. Wood.
\newblock Treewidth of the line graph of a complete graph.
\newblock {\em J. Graph Theory}, 79(1):48--54, 2015.

\bibitem{c6}
Daniel~J. Harvey and David~R. Wood.
\newblock The treewidth of line graphs.
\newblock {\em J. Combin. Theory Ser. B}, 132:157--179, 2018.

\bibitem{c2}
Nina Kamcev, Anita Liebenau, David~R. Wood, and Liana Yepremyan.
\newblock The size {R}amsey number of graphs with bounded treewidth.
\newblock {\em SIAM J. Discrete Math.}, 35(1):281--293, 2021.

\bibitem{Liu&Ciao&Lu}
Ke~Liu, Mengyu Cao, and Mei Lu.
\newblock Treewidth of the Generalized Kneser Graphs.
\newblock {\em Electron. J. Combin.}, 29(1):Paper 1.57, 19, 2022.

\bibitem{c4}
Ke~Liu and Mei Lu.
\newblock The treewidth of 2-section of hypergraphs.
\newblock {\em Discrete Math. Theor. Comput. Sci.}, 23(3):Paper No. 1, 20,
  2021.

\bibitem{Robertson&Seymour}
Neil Robertson and P.~D. Seymour.
\newblock Graph minors. {II}. {A}lgorithmic aspects of tree-width.
\newblock {\em J. Algorithms}, 7(3):309--322, 1986.

\bibitem{c1}
Josse van Dobben~de Bruyn and Dion Gijswijt.
\newblock Treewidth is a lower bound on graph gonality.
\newblock {\em Algebr. Comb.}, 3(4):941--953, 2020.

\bibitem{Wilson}
Richard~M. Wilson.
\newblock The exact bound in the {E}rd{\H{o}}s-{K}o-{R}ado theorem.
\newblock {\em Combinatorica}, 4(2-3):247--257, 1984.

\end{thebibliography}

\end{document}